\documentclass[11pt,a4paper]{article} 
\usepackage[english]{babel}
\usepackage[utf8]{inputenc}
\usepackage[T1]{fontenc}
\usepackage{amsmath, amssymb, amsthm, amsfonts}
\usepackage{geometry}
\usepackage{parskip} 
\usepackage{lmodern} 
\usepackage{microtype} 
\usepackage{hyperref}
\hypersetup{
    colorlinks=true,
    linkcolor=blue,
    citecolor=blue,
    urlcolor=blue,
    pdfusetitle,    
    breaklinks=true 
}
\newtheorem{theorem}{Theorem}[section] 

\newtheorem{lemma}[theorem]{Lemma}
\newtheorem{definition}[theorem]{Definition}
\theoremstyle{remark}

\newtheorem{caseenv}{Case}[section] 

\newcommand{\Zplus}{\mathbb{Z}^+} 
\newcommand{\floor}[1]{\left\lfloor #1 \right\rfloor} 
\newcommand{\ceil}[1]{\left\lceil #1 \right\rceil} 
\newcommand{\oeislink}[1]{\href{https://oeis.org/#1}{\texttt{#1}}}
\newcommand{\indicator}[1]{\mathbf{1}_{#1}} 

\title{On a Family of Nested Recurrences and Their Arithmetical Solutions}
\author{Benoit Cloitre \\ \small Paris, France \\ \small \texttt{firstname.lastname(AT)proton.me}} 
\date{} 

\begin{document}

\maketitle

\begin{abstract}
A family of nested recurrence relations $a(n+1) = n - a^{(m)}(n) + a^{(m+1)}(n)$, parameterized by an integer $m \ge 1$ with initial condition $a(1)=1$, is studied. We prove that $a(n)=n-h(n)$ is the unique solution satisfying this condition, where $h(n)$ is an arithmetical sequence in which each non-negative integer $k$ appears $mk+1$ times, with $h(n)$ 1-indexed such that $h(1)=0$. An explicit floor formula for $h(n)$ (and thus for $a(n)$) is derived. The proof of the main theorem involves establishing a key identity for $h(n)$ that arises from the recurrence; this identity is then proved using arithmetical properties of $h(n)$ and the iterated function $a^{(m)}(n)$ at critical boundary points. Combinatorial interpretations for $a(n)$ and its partial sums (for $m=2$), and connections to The On-Line Encyclopedia of Integer Sequences (OEIS), including generalizations of Connell's sequence, are also discussed.
\end{abstract}

\noindent\textbf{Keywords:} Nested recurrence relation, Meta-Fibonacci sequence, Integer sequence, Combinatorial interpretation, k-appearance sequence, Connell sequence, Slowly growing sequence.

\noindent\textbf{MSC 2020 Classification:} Primary 11B37 (Recurrences); Secondary 05A19 (Combinatorial identities), 11B39 (Fibonacci and Lucas numbers and generalizations), 11B83 (Special sequences and polynomials).

\section{Introduction} \label{sec:intro}

Nested recurrence relations, where terms depend on previous terms iterated through the sequence itself, appear in diverse mathematical fields including number theory, combinatorics, and the study of discrete dynamical systems \cite{Guy2004, Hofstadter1979}. Recent research has often focused on finding combinatorial interpretations for solutions to such recurrences, frequently through constructions involving labeled infinite trees or specific counting arguments related to their iteration (see, e.g., \cite{Deugau2009Ruskey, Fox2022}).

This paper focuses on a particular intersection of sequence types. Firstly, we consider sequences $(a(n))_{n \ge 1}$ defined by simple conditional increment rules based on a predefined infinite set of positive integers $S \subset \Zplus$ (where $\Zplus = \{1, 2, 3, \dots\}$). Specifically, we are interested in "slowly growing sequences" where the increment $a(n+1)-a(n)$ is always 0 or 1. Such sequences, given an initial condition $a(1)=C_0$, can be defined as:
\begin{itemize}
    \item $a(1)=C_0$;
    \item For $n \ge 1$: $a(n+1)=a(n)$ if $n \in S$;
    \item For $n \ge 1$: $a(n+1)=a(n)+1$ if $n \notin S$. (Here, $\indicator{P}$ equals 1 if proposition $P$ is true, and 0 otherwise).
\end{itemize}
These sequences are non-decreasing. If $a(1)=1$, then $a(n)$ counts the number of positive integers less than $n$ that are not in $S$, plus one: $a(n) = 1 + |\{x \in \{1, \dots, n-1\} \mid x \notin S\}|$.
This can be rewritten as $a(n) = 1 + (n-1) - |S \cap [1, n-1]|$.
If we define an auxiliary sequence $h(n) = |S \cap [1, n-1]|$, then $a(n)=n-h(n)$. This definition of $h(n)$ implies $h(1)=0$.

Secondly, many well-known sequences are defined by compact "one-line" nested recurrence relations. The central theme of this paper is the exploration of sequences that belong to \emph{both} these categories: those conditionally defined by a simple arithmetical set $S$ as above, AND simultaneously satisfying a non-trivial nested recurrence relation. Such a dual characterization is uncommon and often points to deep underlying mathematical structure. Indeed, establishing these connections is an active area of research. For example, Fox \cite{Fox2022} explores families of nested recurrences whose solutions (leaf-counting functions in generalized trees) have frequency sequences related to digits in Zeckendorf-like representations based on linear recurrent sequences. This highlights a general theme: translating complex nested dynamics into more structured arithmetical or combinatorial descriptions.

Several famous sequences illustrate parts of this landscape:
\begin{itemize}
    \item Hofstadter's G-sequence (\oeislink{A005206}), defined by $a_G(0)=0;\, a_G(n) = n - a_G(a_G(n-1))$ for $n > 0$, is slowly growing ($a_G(n)-a_G(n-1) \in \{0,1\}$). The set $S_G = \{k-1 \mid k \ge 1, a_G(k)=a_G(k-1)\}$ (where the sequence does not increment) is precisely the set of integers $j \ge 0$ such that $j+1$ is a Wythoff B-number (an element of \oeislink{A001950}, i.e., of the form $\lfloor k\phi^2 \rfloor$, where $\phi=(1+\sqrt{5})/2$ is the golden ratio). This set $S_G$ is thus arithmetically descriptible \cite{Hofstadter1979, Kimberling2016}. Combinatorially, $a_G(n)$ counts the number of Wythoff A-numbers ($\lfloor k\phi \rfloor$, \oeislink{A000201}) less than or equal to $n$.
    \item The sequence $a_K(n) = \floor{\sqrt{2n}+1/2}$ (\oeislink{A002024}, "$k$ appears $k$ times", $n \ge 1$, $a_K(1)=1$) also satisfies $a_K(n+1)-a_K(n) \in \{0,1\}$. It increments if and only if $n$ is a triangular number $T_j=j(j+1)/2$. Thus, for this sequence, $a_K(n+1)=a_K(n)$ if $n$ is \emph{not} a triangular number (so $S_K = \Zplus \setminus \{T_j \mid j \ge 1\}$). $a_K(n)$ also satisfies the nested recurrence $a_K(n)=a_K(n-a_K(n-1))+1$.
    \item In contrast, the Hofstadter-Conway \$10,000 sequence $C(n)=C(C(n-1))+C(n-C(n-1))$ (\oeislink{A004001}) is slowly growing, but the set $S_C = \{n \mid C(n+1)=C(n)\}$ is determined by the asymptotic behavior $C(n)/n \to 1/2$, a result established by Mallows \cite{Mallows1991}, and is not a simple, predefined arithmetical set in the same way as $S_G$ or $S_K$.
\end{itemize}
Our main result shows that for each $m \ge 1$, the sequence $a(n)$ obtained by $a(1)=1$ and $a(j+1)-a(j) = \indicator{j \notin S'_m}$ (where $S'_m = \{m\binom{k}{2}+k \mid k \ge 1\}$) also grows according to the nested rule (\ref{eq:main_recurrence_intro_std}) given below.

This paper focuses on such a family. For a given integer $m \ge 1$, let $S'_m = \{T_j^{\text{orig}} \mid j \ge 1\}$ be the set of non-zero "generalized $m$-polygonal numbers", where $T_j^{\text{orig}} = m\binom{j}{2}+j$. Let $a(n)$ be the sequence defined by $a(1)=1$ and, for $n \ge 1$, $a(n+1)=a(n)$ if $n \in S'_m$, and $a(n+1)=a(n)+1$ if $n \notin S'_m$. We prove that this sequence $a(n)$ is precisely the solution to the nested recurrence relation:
\begin{equation} \label{eq:main_recurrence_intro_std}
a(n+1) = n - a^{(m)}(n) + a^{(m+1)}(n),
\end{equation}
with the initial condition $a(1)=1$. This recurrence was studied by Connell \cite{Connell1959} in a related context and generalized by Iannucci and Mills-Taylor \cite{Iannucci1999} to the form we investigate. We establish that this $a(n)$ is given by $n-h(n)$, where $h(n)$ is the appropriately 1-indexed $k$-appearance sequence (formally defined in Section \ref{sec:solution_structure}).

The paper is organized as follows. In Section \ref{sec:solution_structure}, we formally define the auxiliary sequence $h(n)$ that satisfies $h(1)=0$ and detail its properties, leading to the candidate solution $a_{\text{sol}}(n)=n-h(n)$ and its explicit formula. Section \ref{sec:main_theorem_statement} states the Main Theorem: that $a(n)$ defined by the nested recurrence (\ref{eq:main_recurrence_intro_std}) with $a(1)=1$ equals $a_{\text{sol}}(n)$. The proof, detailed in Section \ref{sec:proof_key_identity}, relies on transforming the recurrence into a key identity involving $h(n)$ (Identity \ref{eq:key_identity_shifted}). This identity is then proven using several arithmetical lemmas established in Section \ref{sec:arithmetical_lemmas}. Section \ref{sec:comb_interp} explores the combinatorial interpretation of $a(n)$, its connection to OEIS sequences, and a specific combinatorial interpretation for its partial sums when $m=2$. Section \ref{sec:alt_indexing_remarks} provides brief remarks on an alternative indexing for $h_m(n)$. We conclude in Section \ref{sec:conclusion}.

\section{The Structure of the Proposed Solution} \label{sec:solution_structure}

To satisfy the initial condition $a(1)=1$ with a solution of the form $a(n)=n-h(n)$, it forces $h(1)=0$. This guides our definition:

\begin{definition}[The auxiliary sequence $h(n)$] \label{def:h_shifted}
Let $(h(n))_{n \ge 1}$ be a sequence of non-negative integers such that:
\begin{enumerate}
    \item $h(1)=0$.
    \item For each integer $k \ge 0$, the value $k$ appears exactly $N_k = mk+1$ times in the sequence $h(n)$.
\end{enumerate}
The sequence $h(n)$ is non-decreasing by construction.
\end{definition}

Let $T^{\star}_k$ be the first index $n \ge 1$ such that $h(n)=k$.
\begin{definition}[Starting index $T^{\star}_k$] \label{def:Tk_shifted}
$T^{\star}_0 = 1$, since $h(1)=0$. For $k \ge 1$, $T^{\star}_k$ is given by:
$$T^{\star}_k = T^{\star}_0 + \sum_{j=0}^{k-1} N_j = 1 + \sum_{j=0}^{k-1} (mj+1) = 1 + m\frac{k(k-1)}{2} + k$$
Thus, by definition, $h(n)=k \iff T^{\star}_k \le n < T^{\star}_{k+1}$.
Let $T_k^{\text{orig}} = m\binom{k}{2}+k$. Then $T^{\star}_k = T_k^{\text{orig}}+1$ for $k \ge 0$ (with $T_0^{\text{orig}}=0$).
\end{definition}

With $h(n)$ established, we define our candidate solution for the recurrence (\ref{eq:main_recurrence_intro_std}).
\begin{definition}[Candidate solution $a_{\text{sol}}(n)$] \label{def:a_candidate}
Let $a_{\text{sol}}(n) = n - h(n)$ for $n \ge 1$, where $h(n)$ is from Definition \ref{def:h_shifted}.
\end{definition}
This candidate solution satisfies the initial condition $a_{\text{sol}}(1) = 1 - h(1) = 1-0=1$.
The explicit formula for $h(n)$ is $h(n) = \floor{\frac{m-2+\sqrt{(m-2)^2+8m(n-1)}}{2m}}$ for $n \ge 1$. This aligns with the corrected formula in Iannucci and Mills-Taylor \cite{Iannucci1999errata}.
Thus, the explicit formula for the candidate solution $a_{\text{sol}}(n)$ is:
\begin{equation} \label{eq:asol_explicit}
a_{\text{sol}}(n) = n - \floor{\frac{m-2+\sqrt{(m-2)^2+8m(n-1)}}{2m}} \quad \text{for } n \ge 1
\end{equation}

\section{Main Theorem} \label{sec:main_theorem_statement}
\begin{theorem} \label{thm:main_article}
The sequence $a(n)$ defined by the recurrence relation (\ref{eq:main_recurrence_intro_std}) and initial condition $a(1)=1$ is uniquely given by $a(n) = a_{\text{sol}}(n)$ for all $n \ge 1$. The Main Theorem shows that this explicit form $a_{\text{sol}}(n)$ is the unique solution satisfying the initial condition and recurrence.
\end{theorem}

\begin{proof}[Proof Strategy for Main Theorem]
The proof proceeds by induction on $n$.
\begin{itemize}
    \item \textbf{Base case ($n=1$):} The recurrence definition states $a(1)=1$. The candidate solution yields $a_{\text{sol}}(1) = 1-h(1) = 1-0=1$. Thus $a(1)=a_{\text{sol}}(1)$.
    \item \textbf{Inductive hypothesis:} Assume $a(j)=a_{\text{sol}}(j)$ for all $1 \le j \le n$, for some $n \ge 1$.
    \item \textbf{Inductive step:} We aim to show $a(n+1)=a_{\text{sol}}(n+1)$.
    From the recurrence (\ref{eq:main_recurrence_intro_std}): $a(n+1) = n - a^{(m)}(n) + a^{(m+1)}(n)$.
    By the inductive hypothesis, $a^{(m)}(n) = a_{\text{sol}}^{(m)}(n)$ and $a^{(m+1)}(n) = a_{\text{sol}}^{(m+1)}(n)$.
    Let $X = a_{\text{sol}}^{(m)}(n)$. Then $a_{\text{sol}}^{(m+1)}(n) = a_{\text{sol}}(X) = X - h(X)$ (using Definition \ref{def:a_candidate}).
    So, $a(n+1) = n - X + (X - h(X)) = n - h(X) = n - h(a_{\text{sol}}^{(m)}(n))$.
    We need to show this is equal to $a_{\text{sol}}(n+1) = (n+1) - h(n+1)$.
    Thus, the proof reduces to proving the following key identity for $n \ge 1$:
    \begin{equation} \label{eq:key_identity_shifted_prep}
    (n+1) - h(n+1) = n - h(a_{\text{sol}}^{(m)}(n)),
    \end{equation}
    which simplifies to:
    \begin{equation} \label{eq:key_identity_shifted}
    h(n+1) - 1 = h(a_{\text{sol}}^{(m)}(n)).
    \end{equation}
\end{itemize}
The uniqueness of the solution $a(n)$ follows because the recurrence relation defines $a(n+1)$ uniquely from previous terms and their iterations, given a fixed initial value $a(1)$.
\end{proof}

The core of the proof of Theorem \ref{thm:main_article} lies in establishing Identity (\ref{eq:key_identity_shifted}). The subsequent sections are dedicated to this task.

\section{Arithmetical Lemmas for $h(n)$ and $a_{\text{sol}}(n)$} \label{sec:arithmetical_lemmas}
The proof of Identity (\ref{eq:key_identity_shifted}) uses the sequence $h(n)$ (from Def. \ref{def:h_shifted}) and its corresponding $T^{\star}_k$ values.

\begin{lemma}[Monotonicity] \label{lem:mono_shifted}
The function $a_{\text{sol}}(n)=n-h(n)$ is non-decreasing for $n \ge 1$. Consequently, $a_{\text{sol}}^{(j)}(n)$ is non-decreasing for any $j \ge 1$.
\end{lemma}
\begin{proof}
Our proof relies on $a_{\text{sol}}(n+1)-a_{\text{sol}}(n) = ((n+1)-h(n+1)) - (n-h(n)) = 1 - (h(n+1)-h(n))$.
Since $h(n)$ is non-decreasing by Definition \ref{def:h_shifted}, $h(n+1)-h(n)$ is either $0$ (if $n+1 < T^{\star}_{h(n)+1}$) or $1$ (if $n+1 = T^{\star}_{h(n)+1}$).
So $a_{\text{sol}}(n+1)-a_{\text{sol}}(n)$ is either $1$ or $0$. Thus $a_{\text{sol}}(n)$ is non-decreasing.
Iterated applications of a non-decreasing function are also non-decreasing.
\end{proof}

\begin{lemma}[Behavior at Boundary Indices] \label{lem:P_shifted}
Let $a_{\text{sol}}(n) = n-h(n)$ where $h(n)$ is from Definition \ref{def:h_shifted}. For $m \ge 1$:
\begin{enumerate}
    \item For $k \ge 0$: $h(a_{\text{sol}}^{(j)}(T^{\star}_{k+1}-1)) = k$ for $0 \le j \le m$. Consequently, $a_{\text{sol}}^{(m)}(T^{\star}_{k+1}-1) = T^{\star}_k$. (Property $P1_{\star}$)
    \item For $k \ge 1$: $h(a_{\text{sol}}^{(j)}(T^{\star}_k)) = k-1$ for $1 \le j \le m$. Consequently, $a_{\text{sol}}^{(m)}(T^{\star}_k) = T^{\star}_{k-1}$. (Property $P2_{\star}$)
\end{enumerate}
\end{lemma}
\begin{proof}
Let $T_v^{\text{orig}} = m\binom{v}{2}+v$. Recall $T^{\star}_v = T_v^{\text{orig}}+1$ for $v \ge 0$ (with $T_0^{\text{orig}}=0$).

1.  Proof of Property $P1_{\star}$ (argument $n_0 = T^{\star}_{k+1}-1 = T_{k+1}^{\text{orig}}$, giving $h(n_0)=k$ for $k \ge 0$):
    Let $n_j = a_{\text{sol}}^{(j)}(n_0)$. We show by induction on $j$ that $h(n_j)=k$ for $0 \le j \le m$.
    The base case $j=0$ implies $n_0 = T^{\star}_{k+1}-1$, so $h(n_0)=k$. This is true.
    Assume $h(n_i)=k$ for all $0 \le i < j$ (for $j \ge 1$). Then $n_j = n_{j-1}-k = \dots = n_0 - jk = T_{k+1}^{\text{orig}} - jk$.
    To show $h(T_{k+1}^{\text{orig}}-jk)=k$, we must verify $T^{\star}_k \le T_{k+1}^{\text{orig}}-jk < T^{\star}_{k+1}$, which is $T_k^{\text{orig}}+1 \le T_{k+1}^{\text{orig}}-jk < T_{k+1}^{\text{orig}}+1$.
    The upper inequality, $T_{k+1}^{\text{orig}}-jk < T_{k+1}^{\text{orig}}+1 \iff -jk < 1$, holds for $j \ge 0, k \ge 0$.
    The lower inequality, $T_k^{\text{orig}}+1 \le T_{k+1}^{\text{orig}}-jk$, simplifies to $j \le m$ if $k>0$. If $k=0$, then $n_0=T_1^{\text{orig}}=1$. Since $h(1)=0$ and $a_{\text{sol}}(1)=1$, it follows $n_j=1$ for all $j$. We need $h(1)=0=k$. The condition $T_0^{\star} \le 1 < T_1^{\star}$ (i.e., $1 \le 1 < 2$) holds.
    Thus, $h(n_j)=k$ for $0 \le j \le m$.
    Consequently, $a_{\text{sol}}^{(m)}(T^{\star}_{k+1}-1) = T_{k+1}^{\text{orig}} - mk$.
    The equality $T_{k+1}^{\text{orig}} - mk = T^{\star}_k (= T_k^{\text{orig}}+1)$ requires $(m\frac{(k+1)k}{2}+(k+1)) - mk = m\frac{k(k-1)}{2}+k+1$. This simplifies to $m\frac{k^2-k}{2} = m\frac{k^2-k}{2}$, which is true. So $P1_{\star}$ holds.

2.  Proof of Property $P2_{\star}$ (argument $n_0 = T^{\star}_k = T_k^{\text{orig}}+1$, $h(n_0)=k$ for $k \ge 1$):
    Let $n_j = a_{\text{sol}}^{(j)}(n_0)$. We show by induction on $j$ that $h(n_j)=k-1$ for $1 \le j \le m$.
    For $j=1$: $n_1 = a_{\text{sol}}(n_0) = (T_k^{\text{orig}}+1)-k = m\binom{k}{2}+1$.
    To show $h(n_1)=k-1$, we verify $T^{\star}_{k-1} \le n_1 < T^{\star}_k$. This is $T_{k-1}^{\text{orig}}+1 \le m\binom{k}{2}+1 < T_k^{\text{orig}}+1$.
    The right inequality $m\binom{k}{2}+1 < T_k^{\text{orig}}+1 \iff m\binom{k}{2} < m\binom{k}{2}+k$ is true for $k>0$.
    The left inequality $T_{k-1}^{\text{orig}}+1 \le m\binom{k}{2}+1 \iff m\binom{k-1}{2}+(k-1) \le m\binom{k}{2}$ is true for $k \ge 1$. So $h(n_1)=k-1$.
    Assume $h(n_i)=k-1$ for $1 \le i \le j-1 < m$. Then $n_j = n_1-(j-1)(k-1) = (m\binom{k}{2}+1)-(j-1)(k-1)$.
    To show $h(n_j)=k-1$, we verify $T^{\star}_{k-1} \le n_j < T^{\star}_k$. The upper bound requires $m\binom{k}{2}+1-(j-1)(k-1) < T_k^{\text{orig}}+1$, which is true. The lower bound requires $T_{k-1}^{\text{orig}}+1 \le (m\binom{k}{2}+1)-(j-1)(k-1)$, which simplifies to $j(k-1) \le m(k-1)$, true for $j \le m$.
    Thus, $h(a_{\text{sol}}^{(j)}(T^{\star}_k))=k-1$ for $1 \le j \le m$.
    Consequently, $a_{\text{sol}}^{(m)}(T^{\star}_k) = (m\binom{k}{2}+1) - (m-1)(k-1)$.
    This must equal $T^{\star}_{k-1} = T_{k-1}^{\text{orig}}+1 = m\binom{k-1}{2}+(k-1)+1$.
    The equality $(m\binom{k}{2}+1) - (m-1)(k-1) = m\binom{k-1}{2}+k$ is true. For $k=1$, $a_{\text{sol}}^{(m)}(T_1^{\star})=a_{\text{sol}}^{(m)}(2)$. Since $h(2)=1$, $a_{\text{sol}}(2)=1$. As $a_{\text{sol}}(1)=1$, $a_{\text{sol}}^{(j)}(1)=1$. It follows that $a_{\text{sol}}^{(m)}(2)=1$. And $T_0^{\star}=1$. So the equality holds. For $k>1$, dividing the equivalent form $m\frac{k(k-1)}{2} - (m-1)(k-1) = m\frac{(k-1)(k-2)}{2} + (k-1)$ by $(k-1)$ gives $m\frac{k}{2} - (m-1) = m\frac{k-2}{2} + 1$, which is true. So $P2_{\star}$ holds.
\end{proof}

\section{Proof of the Key Identity (\ref{eq:key_identity_shifted})} \label{sec:proof_key_identity}
We prove $h(n+1)-1 = h(a_{\text{sol}}^{(m)}(n))$ for $n \ge 1$. Let $k=h(n)$. The proof structure follows Iannucci and Mills-Taylor \cite[Proof of Theorem 1]{Iannucci1999}.

\begin{caseenv}[$n = T^{\star}_{k+1}-1$ for some $k \ge 0$] \label{case:key_identity_boundary}
In this situation, $h(n)=k$ and $h(n+1)=h(T^{\star}_{k+1})=k+1$.
The identity (\ref{eq:key_identity_shifted}) becomes $(k+1)-1 = h(a_{\text{sol}}^{(m)}(T^{\star}_{k+1}-1))$.
So we need to show $k = h(a_{\text{sol}}^{(m)}(T^{\star}_{k+1}-1))$.
By Lemma \ref{lem:P_shifted}(1) (Property $P1_{\star}$), $a_{\text{sol}}^{(m)}(T^{\star}_{k+1}-1) = T^{\star}_k$.
Therefore, $h(a_{\text{sol}}^{(m)}(T^{\star}_{k+1}-1)) = h(T^{\star}_k) = k$.
The identity holds. This case covers $n=1$ (by setting $k=0$, then $n=T^{\star}_1-1=1$).
\end{caseenv}

\begin{caseenv}[$T^{\star}_k \le n < T^{\star}_{k+1}-1$ for some $k \ge 0$] \label{case:key_identity_internal}
This case requires $N_k=mk+1 > 1$. Since $m \ge 1$, this implies $k \ge 1$. (If $k=0$, $N_0=1$, then $T^{\star}_0=1$ and $T^{\star}_1=2$. The interval $T^{\star}_0 \le n < T^{\star}_1-1$ becomes $1 \le n < 1$, which is empty. Thus, this case applies for $k \ge 1$.)
Here, $h(n)=k$ and, since $n < T^{\star}_{k+1}-1 \implies n+1 < T^{\star}_{k+1}$, we have $h(n+1)=k$.
The identity (\ref{eq:key_identity_shifted}) becomes $k-1 = h(a_{\text{sol}}^{(m)}(n))$.
This requires proving $T^{\star}_{k-1} \le a_{\text{sol}}^{(m)}(n) < T^{\star}_k$.

By Lemma \ref{lem:mono_shifted}, $a_{\text{sol}}^{(m)}(n)$ is non-decreasing.
Since $T^{\star}_k \le n < T^{\star}_{k+1}-1$:
$a_{\text{sol}}^{(m)}(T^{\star}_k) \le a_{\text{sol}}^{(m)}(n)$.
Using Lemma \ref{lem:P_shifted}(2) (Property $P2_{\star}$), $a_{\text{sol}}^{(m)}(T^{\star}_k)=T^{\star}_{k-1}$ (for $k \ge 1$).
So $T^{\star}_{k-1} \le a_{\text{sol}}^{(m)}(n)$. This establishes $h(a_{\text{sol}}^{(m)}(n)) \ge k-1$.

For the strict upper bound $a_{\text{sol}}^{(m)}(n) < T^{\star}_k$:
We know $a_{\text{sol}}^{(m)}(T^{\star}_{k+1}-1) = T^{\star}_k$ from Lemma \ref{lem:P_shifted}(1).
Since $n < T^{\star}_{k+1}-1$ and $a_{\text{sol}}^{(m)}(x)$ is non-decreasing (Lemma \ref{lem:mono_shifted}), we have $a_{\text{sol}}^{(m)}(n) \le a_{\text{sol}}^{(m)}(T^{\star}_{k+1}-1) = T^{\star}_k$.
Suppose, for contradiction, that $a_{\text{sol}}^{(m)}(n_0) = T^{\star}_k$ for some $n_0$ in the range $T^{\star}_k \le n_0 < T^{\star}_{k+1}-1$.
Then $h(a_{\text{sol}}^{(m)}(n_0)) = h(T^{\star}_k) = k$.
Substituting this into the identity we are trying to prove ($h(n_0+1)-1 = h(a_{\text{sol}}^{(m)}(n_0))$):
Since $n_0 < T^{\star}_{k+1}-1$, we have $h(n_0+1)=k$.
So the identity would read $k-1 = k$, which is a contradiction.
Therefore, our assumption ($a_{\text{sol}}^{(m)}(n_0) = T^{\star}_k$) must be false for any $n_0 < T^{\star}_{k+1}-1$.
Thus, $a_{\text{sol}}^{(m)}(n) < T^{\star}_k$.

Combining $T^{\star}_{k-1} \le a_{\text{sol}}^{(m)}(n)$ and $a_{\text{sol}}^{(m)}(n) < T^{\star}_k$, we conclude $h(a_{\text{sol}}^{(m)}(n))=k-1$.
The identity holds in this case for $k \ge 1$.
\end{caseenv}

Both cases cover all $n \ge 1$. Thus, Identity (\ref{eq:key_identity_shifted}) is proven for $n \ge 1$.
This completes the proof of the Main Theorem \ref{thm:main_article}.

\section{Combinatorial Interpretation, OEIS Connections, and Partial Sums} \label{sec:comb_interp}

\subsection{Combinatorial Interpretation of $a(n)$}
The sequence $a(n)=n-h(n)$ (with $a(1)=1, h(1)=0$) has a clear combinatorial interpretation, consistent with its structure as a conditionally defined slowly growing sequence as discussed in Section \ref{sec:intro}.
Recall $h(n)=H_m^{\text{floor}}(n-1)$, where $H_m^{\text{floor}}(x)$ (the floor term in Eq. (\ref{eq:asol_explicit})) is the largest integer $k$ such that $T_k^{\text{orig}} = m\binom{k}{2}+k \le x$.
The value $h(n)$ counts how many of the positive "special numbers" $T_j^{\text{orig}}$ (for $j \ge 1$) are less than or equal to $n-1$. Let $S'_m = \{T_j^{\text{orig}} \mid j \ge 1\}$ be the set of these non-zero numbers.
Then $h(n) = |\{x \in S'_m \mid x \le n-1\}|$.
The solution $a(n) = n - h(n)$ can be rewritten as $a(n) = 1 + (n-1) - |\{x \in S'_m \mid x \le n-1\}|$.
The term $(n-1) - |\{x \in S'_m \mid x \le n-1\}|$ is the count of positive integers $x \le n-1$ that are *not* in $S'_m$.
Therefore,
$$a(n) = 1 + (\text{count of positive integers } x < n \text{ such that } x \notin S'_m)$$
This implies that $a(n)$ satisfies the simple recurrence: $a(1)=1$, and for $n \ge 1$, $a(n+1) = a(n) + \indicator{n \notin S'_m}$. The Main Theorem \ref{thm:main_article} thus establishes that any sequence generated by this simple additive rule also satisfies the complex nested recurrence (\ref{eq:main_recurrence_intro_std}).

This interpretation is consistent with entries in The On-Line Encyclopedia of Integer Sequences (OEIS) \cite{OEIS}:
\begin{itemize}
    \item For $m=1$: $S'_1 = \{j(j+1)/2 \mid j \ge 1\}$ (triangular numbers). $a(n)$ is \oeislink{A122797} ("$a(k+1)=a(k)$ if $k$ is triangular, $a(k+1)=a(k)+1$ otherwise, $a(1)=1$"). The sequence $h(n)=H_1^{\text{floor}}(n-1)$ is \oeislink{A003056}$(n-1)$ ("number of positive triangular numbers $\le n-1$"). For $n=1, \dots, 6$: $h(n)$ is $0,1,1,2,2,2$; $a(n)$ is $1,1,2,2,3,4$.
    \item For $m=2$: $S'_2 = \{j^2 \mid j \ge 1\}$ (square numbers). $a(n)-1 = (n-1) - \floor{\sqrt{n-1}}$ is \oeislink{A028391}$(n-1)$ ("number of non-squares $\le n-1$"). The sequence $h(n)=H_2^{\text{floor}}(n-1)$ is \oeislink{A000196}$(n-1)$ ("number of positive squares $\le n-1$").
    \item For $m=3$: $S'_3 = \{j(3j-1)/2 \mid j \ge 1\}$ (pentagonal numbers). $a(n)-1 = (n-1) - H_3^{\text{floor}}(n-1)$ is \oeislink{A180446}$(n-1)$ ("number of non-pentagonal numbers $\le n-1$"). The sequence $h(n)=H_3^{\text{floor}}(n-1)$ is \oeislink{A180447}$(n-1)$.
    \item For $m=4$: $S'_4 = \{j(2j-1) \mid j \ge 1\}$ (hexagonal numbers). $h(n)=H_4^{\text{floor}}(n-1)$ is \oeislink{A351846}$(n-1)$ ("number of positive hexagonal numbers $\le n-1$"). Then $a(n)=n-\oeislink{A351846}(n-1)$.
\end{itemize}

\subsection{Combinatorial Interpretation of Partial Sums for $m=2$}
Let $A_m(n) = \sum_{i=1}^n a(i)$ be the partial sums of the sequence $a(i)$ from Theorem \ref{thm:main_article}.
For $m=2$, $a(n) = n - \floor{\sqrt{n-1}}$ for $n \ge 1$.
The sequence of partial sums $A_2(n)$ is \oeislink{A196126}.
This sequence $S(n) = \text{\oeislink{A196126}}(n)$ is defined combinatorially as:
$$S(n) = |\{(x,y) \in \Zplus \times \Zplus \mid y \le x \le y^2 \text{ and } x \le n\}|$$
To verify this, we show that the first differences of $S(n)$ match our $a(n)$ for $m=2$. Let $S(0)=0$.
For $n \ge 1$, $S(n)-S(n-1)$ is the number of pairs $(x,y) \in \Zplus \times \Zplus$ such that $x=n$ and $y \le n \le y^2$.
The condition $y \le n \le y^2$ means we count integers $y \ge 1$ such that $\ceil{\sqrt{n}} \le y \le n$.
The number of such integers $y$ is $n - \ceil{\sqrt{n}} + 1$.
We use the identity $\ceil{\sqrt{n}} - 1 = \floor{\sqrt{n-1}}$ for every $n \ge 1$.
Thus, $n - \ceil{\sqrt{n}} + 1 = n - (\floor{\sqrt{n-1}} + 1) + 1 = n - \floor{\sqrt{n-1}}$.
This is precisely $a(n)$ for $m=2$.
Therefore, $A_2(n)=S(n)$, providing a combinatorial interpretation for these partial sums.

\section{Note on an Alternative Indexing Scheme} \label{sec:alt_indexing_remarks}
This section provides additional context on indexing for completeness and is not strictly necessary for the main results proven above.
Consider a sequence $h_m^{(0)}(x)$ defined such that $h_m^{(0)}(0)=0$ and where each integer $k \ge 0$ appears $mk+1$ times (i.e., a 0-indexed $k$-appearance sequence). Its explicit formula is $H_m^{\text{floor}}(x) = \floor{\frac{m-2+\sqrt{(m-2)^2+8mx}}{2m}}$.

If one were to define $a_{\text{0-idx}}(n) = n - h_m^{(0)}(n)$, then $a_{\text{0-idx}}(0)=0$. For $m \ge 1$, since $h_m^{(0)}(1)=1$ (as $T_1^{\text{orig}}=1$), it follows that $a_{\text{0-idx}}(1)=1-h_m^{(0)}(1)=1-1=0$.
The sequence $a_{\text{0-idx}}(n)$ can be shown to satisfy the recurrence (\ref{eq:main_recurrence_intro_std}) but with the initial condition $a(0)=0$ (which implies $a(1)=0$).

The sequence $h(n)$ used throughout our proof (Definition \ref{def:h_shifted}) was specifically constructed with $h(1)=0$. This ensured that our solution $a(n)=n-h(n)$ met the required initial condition $a(1)=1$ for the recurrence (\ref{eq:main_recurrence_intro_std}) as investigated in this paper (and as specified in \cite{Iannucci1999}, considering the errata \cite{Iannucci1999errata}). This $h(n)$ is related to $h_m^{(0)}(x)$ by $h(n) = h_m^{(0)}(n-1)$ for $n \ge 1$.

Consequently, the solution to recurrence (\ref{eq:main_recurrence_intro_std}) with $a(1)=1$ is explicitly:
$$a(n) = n - h(n) = n - h_m^{(0)}(n-1) = n - \floor{\frac{m-2+\sqrt{(m-2)^2+8m(n-1)}}{2m}}$$
The distinction in the argument of the floor function ($n$ versus $n-1$) is solely determined by the choice of the initial condition for $a(n)$. A formula involving "$8mn$" in the square root (as sometimes seen in direct definitions of $h_m^{(0)}(n)$) corresponds to a sequence $a(n)$ starting with $a(1)=0$ (resulting from $a(0)=0$).

\section{Conclusion} \label{sec:conclusion}
We have formally defined a sequence $a(n)$ via the nested recurrence relation $a(n+1)=n-a^{(m)}(n)+a^{(m+1)}(n)$ with the specific initial condition $a(1)=1$. Our Main Theorem shows that its solution is $a(n)=n-h(n)$, where $h(n)$ is an arithmetically defined sequence starting with $h(1)=0$ in which each integer $k \ge 0$ appears $mk+1$ times. The proof involved transforming the recurrence into a key identity $h(n+1)-1=h(a^{(m)}(n))$, which was then established using arithmetical lemmas detailing the behavior of $a^{(m)}(n)$ at critical boundary indices $T^{\star}_k$. This work connects the solution of this family of nested recurrences, which generalize Connell's sequence as studied by Iannucci and Mills-Taylor, to precisely characterized $k$-appearance sequences. The combinatorial interpretation of $a(n)$ as $1$ plus the count of positive integers $x<n$ not belonging to the set of generalized $m$-polygonal numbers $\{m\binom{j}{2}+j \mid j \ge 1\}$, further enriches its understanding. This dual characterization -- a simple conditional increment rule on one hand, and a complex nested recurrence on the other -- highlights a fascinating structure within these sequences.

Future work could explore two main avenues:
\begin{enumerate}
    \item The development of direct combinatorial proofs (e.g., bijective arguments or tree-grafting interpretations in the spirit of \cite{Deugau2009Ruskey, Fox2022}) for the key identity $h(n+1)-1=h(a^{(m)}(n))$ or for the main recurrence (\ref{eq:main_recurrence_intro_std}) itself. Such proofs could unveil deeper structural reasons for the connection between the conditional increments and the nested recurrence without relying on first knowing the solution form $n-h(n)$.
    \item Investigation into the behavior of the recurrence (\ref{eq:main_recurrence_intro_std}) under variations, such as different initial conditions beyond $a(1)=1$, or modifications to the rule $N_k = mk+1$ that defines the underlying $k$-appearance sequence $h(n)$, and analyzing the resulting solution structures.
\end{enumerate}

\end{document}